\documentclass{elsart}

\usepackage{amssymb}

\usepackage{psfig}
\usepackage{epsfig}
\usepackage{graphicx}
\usepackage{amsmath}
\usepackage{amscd}

\def\ds{\displaystyle }

\def\fc{\hbox{\rm l\hskip -6pt 1}}

\def\M{{\bf M}}
\def\p{{\bf p}}
\def\q{{\bf q}}
\def\b{{\bf b}}
\def\d{{\bf d}}

\def\h{{\bf h}}
\def\r{{\bf r}}
\def\w{{\bf w}}
\def\z{{\bf z}}

\def\bgamma{{\boldsymbol{\gamma}}}
\def\bphi{{\boldsymbol{\phi}}}
\def\bpsi{{\boldsymbol{\psi}}}
\def\Q{{\mathcal{Q}}}
\def\O{{\mathcal{O}}}
\def\HH{{\bf H}}
\def\H1{{\bf H}^1}

\newtheorem{lemma}{Lemma}[section]

\newtheorem{theorem}{Theorem}[section]

\newenvironment{proof}
{\noindent {\bf Proof:}}{$\square$ \\}

\begin{document}

\begin{frontmatter}


 \title{A multi-region nonlinear age-size structured fish population model}

 \author{Blaise Faugeras\corauthref{cor1}}
 \ead{Blaise.Faugeras@ifremer.fr}
 \corauth[cor1]{}
 \author{and Olivier Maury}

 \address{IRD, CRHMT, Av. Jean Monnet, BP 171, 34200 S\`ete, France}

\begin{abstract}
The goal of this paper is to present a generic multi-region
nonlinear age-size structured fish population model, and to
assess its mathematical well-posedness. An initial-boundary value
problem is formulated. Existence and uniqueness of a positive weak
solution is proved. Eventually, a comparison result is derived : the population
of all regions decreases as the mortality rate increases in at least one region.
\end{abstract}

\end{frontmatter}

\section{Introduction}
Fish population dynamics models are essential to provide
assessment of the fish abundance and fishing pressure. Their use
forms the basis of scientific advice for fisheries managements.
Discrete age structured models are most of the time used for fisheries
stock assessments \cite{Megrey:1989}. Indeed, ecologists,
mathematicians and population biologists have observed that the
age structure provides more realistic results at reasonable
computational expense for a wide variety of biological populations
(see \cite{Webb:1985}, \cite{DeAngelis:1993}, \cite{Swart:1994}, \cite{Arino:1995}).\\
In this paper we study a model which was first designed to
represent Atlantic bigeye tuna populations \cite{Maury:2004} but
which is also generic enough to be potentially usefull for various
fish species. Indeed, most fish populations share specific
characteristics which have to be taken into account in order
to model their dynamics in a realistic manner.\\
A first point concerning tuna fisheries is that they are highly
heterogeneous in space and time. This has an important impact on
their functioning. Important migrations of fish occur at various
scales and fish movements have to be explicitly represented.
Moreover, growth potentially varies with space that is to say with
the region of the ocean under consideration. Hence, fishes of the
same age can exhibit very different sizes depending of their
various history. Consequently a spatialized approach taking
explicitly into account the potential variability of growth in
space
has to be used.\\
A second point is that, because of non-uniform mortality over
sizes, bias on both growth and mortality estimates may result from
simply adding a gaussian size distribution to an age structured
model as it is generally done. It is reasonable to think that the
use of both age and size as structure variables
should enable to overcome this difficulty.\\
These are some of the principal problems of current stock
assessment models. That is why it is necessary to carry on the
modelling effort by proposing and testing more complex models.
This paper follows
this direction and its purpose is twofold.\\
First we describe a synthetic and generic model of population
dynamics in which both age and size are taken as structure
variables and in which fish movements among spatial regions are
explicitly represented. The model is a system of coupled partial
differential equations. Nonlocal nonlinearities appear in the
boundary conditions modelling recruitment that is to say the birth
law or density dependent fish reproduction. The relative
complexity of the model enables a direct and simultaneous comparison with all
the data available for tuna fisheries such as catches, fishing
efforts, size frequencies, tagging data, and otoliths increments.
This paper does not aim at getting into all the details of the
parameterizations used to represent a particular tuna population
and we refer to \cite{Maury:2004} for these points.\\
Our second and most important goal is to assess the mathematical
well-posedness of the model. The paper is organized as follows.
The equations of the model are presented in Section
\ref{sec:present}. Sections \ref{sec:preliminaries},
\ref{section:exist} and \ref{section:positive} deal with the
mathematical analysis of the model. In Section
\ref{sec:preliminaries} we formulate an initial-boundary value
problem, introduce a variational formulation and state our main
mathematical results. Existence of a unique weak solution is shown
in Section \ref{section:exist}. As often with nonlinear problems
the proof uses a fixed point argument. The methodology follows the
one proposed in \cite{Langlais:1985} for a scalar equation. It has
to be adapted in order to be able to deal with our nonlinear
system. We also show positivity of the solution and give a
comparison result in Section \ref{section:positive}. Namely we
prove that if the fish mortality rate increases in at least one
geographic region then the population globally decreases in all
regions.


\section{The model}
\label{sec:present} The dynamics of the population of fish is
described through density functions $p_i(t,a,l)$ where time $t \in
(0,T)$, age $a \in (0,A)$ and length $l \in (0,L)$ are continuous
variables and the subscript $i \in [1:N]$ refers to the geographic
zone or region under consideration. The number of fish of age
between $a_1$ and $a_2$, of length between $l_1$ and $l_2$ at time
$t$ in region $i$ is given by the integral
\begin{equation*}
\ds \int_{a_1}^{a_2} \int_{l_1}^{l_2} p_i(t,a,l)dl da,
\end{equation*}

Let us set ${\mathcal{O}}=(0,T)\times(0,A)$ and
${\mathcal{Q}}={\mathcal{O}} \times (0,L)$. The time evolution of
the population given by Eq. \ref{eqn:e1-1} includes the following
processes.

In region $i$, as time goes on and fishes grow older, their length
increases with a growth rate $\gamma_i$.
In a fish population individuals of the same age can often differ
markedly in size \cite{Pfister:2002}. This variability in growth
can result from many different mechanisms, including genetic or
behavorial traits that confer different performances to
individuals, and factors such as environmental heterogeneity and
variability \cite{Beverton:1996}. In fishery science, this
variability is usually taken into account in age-structured models
using a length-at-age relation perturbed by a Gaussian noise with
a length dependent standard deviation (see for example
\cite{Fournier:1990}). The model discussed here is
length-structured and uses a diffusion term in the length variable
with dispersion rate $d_i$ to account for individuals having the
same age but different lengths. The advection-diffusion term in
length can be seen as the limit of a random walk model in which
each individual grows with an average velocity, but has at each
time step a small binomial probability to grow faster or slower
than this average (see the book by Okubo \cite{Okubo:1980} for
more details).

The model also describes mortality and migration of individuals.
The mortality rate is split into natural mortality $\mu_i$ and
fishing mortality $f_i$. Let also $m_{i\rightarrow j}$ be the
migration rate of individuals going
from region $i$ to region $j$ ($m_{i\rightarrow j}=0$ if regions $i$ and $j$ are not adjacent).\\
The density functions $p_i$ for $i \in [1:N]$ follow the balance
law:
\begin{equation}
\label{eqn:e1-1} \left \lbrace \begin{array}{ll}
\partial_t p_i(t,a,l) + \partial_a p_i(t,a,l)=&
\partial_l(d_i(t,a,l) \partial_l p_i(t,a,l)) - \partial_l(\gamma_i(t,a,l) p_i(t,a,l))\\
&+ \ds \sum_{j \ne i}^N m_{j \rightarrow i}(t,a,l) p_j(t,a,l)\\
&- (\ds \sum_{j \ne i}^N m_{i \rightarrow j}(t,a,l) ) p_i(t,a,l)\\
& -(\mu_i(t,a,l) + f_i(t,a,l))p_i(t,a,l),\quad (t,a,l) \in
\mathcal{Q},
\end{array}
\right.
\end{equation}
These equations have to be completed with initial and boundary conditions.\\
Homogeneous Neumann boundary conditions at $l=0$ and $l=L$ express
the fact that the length of individuals can not reach negative
values or values larger than $L$.
\begin{equation}
\label{eqn:e1-4}
\partial_l p_i(t,a,0)=\partial_l p_i(t,a,L)=0,\quad (t,a)\in \mathcal{O}.
\end{equation}
The initial age and size distribution is prescribed,
\begin{equation}
\label{eqn:e1-2} p_i(0,a,l)=p_i^0(a,l),\quad (a,l) \in
(0,A)\times(0,L).
\end{equation}
We also need a boundary condition for $a=0$ that is to say a
recruitment law. It is written as:
\begin{equation}
\label{eqn:e1-3} p_i(t,0,l)=\beta_i(t,l,P_i(t)),\quad (t,l)\in
(0,T)\times(0,L),
\end{equation}
The length of recruited fish is assumed to lie between $0$ and a
small constant length $L_b$. Moreover we denote by $L_m$ the
minimal length of fishes which have reached maturity. $L_b$ and
$L_m$ satisfy $0<L_b<L_m<L$. The stock spawning biomass is
calculated as
\begin{equation}
\label{eqn:e2} P_i(t)=\ds \int_0^A \int_{L_m}^L
w_i(t,a,l)p_i(t,a,l)dlda,
\end{equation}
where $w_i$ is a weighting function. Finally we use a Beverton and
Holt \cite{Beverton:1996} stock-recruitment relation in each
region and obtain,
\begin{equation}
\label{eqn:e3} \beta_i(t,l,P)=\fc_{[0,L_b]}(l) \psi_i(t) \ds
\frac{P}{\theta_i + P},
\end{equation}
where  $\fc_{[0,L_b]}$ is the usual characteristic function,
$\theta_i > 0$ is a constant parameter and $\psi_i(t)$ is a given
function of time used to parameterize fluctuations of the
recruitment not taken into account in the Beverton and Holt
relation.

\section{Main assumptions and preliminary results}
\label{sec:preliminaries}

In this section we set the mathematical frame in which the
analysis is conducted. We formulate the main assumptions which are
made on the data of the model, give the definition of a weak
solution to the initial-boundary value problem and state our
results in Theorems \ref{theo:exist-unique} and
\ref{theo:comparaison}.

\subsection{Functional spaces}
Let us introduce the functional spaces which we use in the remainder of this work.\\
The vectorial notation $\p =(p_1,...,p_N)^T$ is used. The usual
scalar product of two vectors $\p,\q \in \mathbb{R}^N$ is denoted
by $\p.\q$ and the norm of
$\p$ by $|\p|$.\\
$\HH$ and $\H1$ are the separable Hilbert spaces defined by $\HH
=(L^2(0,L))^N$ and $\H1 = (H^1(0,L))^N$. $\HH$ is equipped with
the scalar product
$$
(\p,\q)_{\HH}= \ds \int_{0}^{L} \p(l).\q(l)dl.
$$
We denote by $||.||_{\HH}$ the induced norm on ${\bf H}$.\\
$\H1$ is equipped with the scalar product
$$(\p,\q)_{\H1}= \ds
\int_{0}^{L} \p(l).\q(l)dl+ \displaystyle \int_{0}^{L}
\partial_l \p(l).\partial_l \q(l)dl.
$$
We denote by $||.||_{\H1}$, the induced norm on $\H1$.\\
By $<.,.>$ we denote the duality between $\H1$ and its dual $(\H1)'$.\\
$L^2(\O,\HH)$ (resp. $L^2(\O,\H1)$) denotes the Hilbert space of
measurable functions of $\O$ with values in $\HH$ (resp. $\H1$)
such that\\
$||\p||_{L^2(\O,\HH)} = ( \ds \int_\O ||\p(t,a,.)||_{\HH}^2 dt da)^{1/2} < \infty$
(resp. $||\p||_{L^2(\O,\H1)} < \infty$). \\
We also make use of the notation $V=L^2(\O,\H1)$ and
the dual space $V'=L^2(\O,(\H1)')$. By $<<.,.>>$ we denote the duality between $V$ and its dual $V'$.\\
The partial derivatives $\partial_t$ and $\partial_a$ denote
differentiation in ${\mathcal{D}}'(\O,(\H1)')$ and $D$ stands for
$\partial_t + \partial_a$.\\
We will have to use he following trace result.
\begin{lemma}
\label{lemma:trace}
Let $\p,\q \in V$ such that $D\p,D\q \in V'$. It holds that:\\
For all $t_0 \in (0,T)$ and all $a_0 \in (0,A)$, $\p$ has a trace
at $t=t_0$ belonging to $(L^2((0,A)\times(0,L)))^N$ and at $a=a_0$
belonging to $(L^2((0,T)\times(0,L)))^N$. The trace applications
are continuous in the strong and weak topology. Moreover the
following integration by parts formula holds,
$$
\begin{array}{ll}
\ds \int_\O [<D\p,\q>+<D\q,\p>] dtda=&\ds \int_0^A \int_0^L [\p.\q(T,a,l)-\p.\q(0,a,l)]dadl\\[7pt]
&+ \ds \int_0^T \int_0^L [\p.\q(t,A,l)-\p.\q(t,0,l)]dtdl
\end{array}
$$
\end{lemma}
\begin{proof}
This result is the extension to dimension $N$ of Lemma 0 in
\cite{Garroni:1982}. Also see \cite{Lions:1968}.
\end{proof}

\noindent We will also have to consider the space ${\bf
L}^{\infty} = (L^{\infty}(\Q))^N$. $L^{\infty}(\Q)$ is a Banach
space equipped with the norm $||p_i||_{\infty} = {\mathrm{inf}}
\lbrace M;|p_i(t,a,l)| \le M \ {\mathrm{a.e.}} \ {\mathrm{in}} \
\Q \rbrace$. Similarly ${\bf L}^{\infty}$ is a Banach space
equipped with the norm $||\p||_{\infty} =
\underset{i\in[1:N]}{max}||p_i||_{\infty}$.

\subsection{Assumptions on the data and preliminary transformation of the system}

The movements rates $m_{i \rightarrow j}$ are assumed to satisfy
\begin{itemize}
\item $m_{i \rightarrow j}(t,a,l) \ge 0$ a.e in $\Q$, $m_{i \rightarrow j} \in L^{\infty}(\Q)$.
\end{itemize}
We define the matrix of movements $\M$ by
$$
M_{ij}= \left \lbrace \begin{array}{l}
m_{j \rightarrow i} \quad {\mathrm{if}}\ i \ne j,\\
- \ds \sum_{k \ne i}^{N} m_{i \rightarrow k} \quad {\mathrm{if}}\ i = j.\\
\end{array}
\right.
$$
Hence the term, $[\ds \sum_{j \ne i}^N m_{j \rightarrow i} p_j
- (\ds \sum_{j \ne i}^N m_{i \rightarrow j} ) p_i]$, in Eq. \ref{eqn:e1-1} can be written in matrix form as $(\M\p)_i$.\\
\noindent Concerning the diffusion coefficients $d_i$, the growth
rates $\gamma_i$ and the natural and fishing mortality rates
$\mu_i$ and $f_i$, we make the following assumptions for all $i
\in [1:N]$:
\begin{itemize}
\item $d_i(t,a,l) \ge d^0 > 0$, a.e in $\Q$, $d_i \in L^{\infty}(\Q)$,
\item $\gamma_i(t,a,l)$ is differentiable with respect to $l$, and $\gamma_i, \partial_l \gamma_i \in L^{\infty}(\Q)$,
\item $\mu_i(t,a,l),f_i(t,a,l) \ge 0$, a.e in $\Q$, $\mu_i,f_i \in L^{\infty}(\Q)$. We also make use of the
notation $z_i=\mu_i+f_i$.
\end{itemize}
\noindent In the formulation of the recruitment process (Eq.
\ref{eqn:e1-3}) $\psi_i$ and $w_i$ satisfy:
\begin{itemize}
\item $\psi_i(t) \ge 0$ a.e in $(0,T)$ and $\psi_i \in L^{\infty}(0,T)$,
\item $w_i(t,a,l) \ge 0$ a.e in $\Q$ and $w_i \in L^{\infty}(\Q)$.
\end{itemize}
\noindent The initial distributions $p_i^0(a,l)$ satisfies for all
$i \in [1:N]$:
\begin{itemize}
\item $p_i^0(a,l) \ge 0$ a.e in $\Q$, $p_i^0 \in L^{2}((0,A)\times(0,L))$.
\end{itemize}
\noindent In order to prove our existence result it is convenient
to perform a change of unknown function: $\p$ satisfies
(\ref{eqn:e1-1})-(\ref{eqn:e1-3}) if and only if
$\hat{\p}=e^{-\lambda t}\p$ is a solution to the same system where
$-(\mu_i+f_i)p_i$ is replaced $-(\mu_i+f_i+\lambda)p_i$ in Eq.
\ref{eqn:e1-1} and $\beta_i$ in the expression of the boundary
condition at $a=0$ (Eq. \ref{eqn:e1-3}) is replaced by
\begin{equation}
\label{eqn:e4} \hat{\beta}_i(t,l,\hat{P}_i(t))= \fc_{[0,L_b]}(l)
\psi_i(t) \ds \frac{\hat{P}_i(t)}{\theta_i e^{-\lambda t} +
\hat{P}_i(t)},
\end{equation}
\begin{equation}
\label{eqn:e5} \hat{P}_i(t)=\ds \int_0^A \int_{L_m}^L
w_i(t,a,l)\hat{p}_i(t,a,l)dlda.
\end{equation}
In the remaining part of this paper this change of unknown is
implicitly done and we omit the $\hat{p_i}$ notation. The constant
$\lambda$ will be fixed to a convenient value below. Moreover, the
possible nullification of the term $\theta_i e^{-\lambda t} +
\hat{P}_i(t)$, invites us to define,
\begin{equation}
\label{eqn:e6} \beta_i(t,l,P_i(t))= \fc_{[0,L_b]}(l) \psi_i(t) \ds
\frac{P_i(t)}{\theta_i e^{-\lambda t} + |P_i(t)|}.
\end{equation}
This formulation will be used in the following. We will show that
if initial distributions, $p_i^0$ are nonnegative then $p_i \ge 0$
a.e. in $\Q$, thus the two formulations are equivalent.

\subsection{Variational formulation and weak solutions}

Formally multiplying Eq. \ref{eqn:e1-1} by a function $q_i$ and
integrating by parts on $(0,L)$ yields to the definition of the
following linear forms. For $p_i,q_i \in H^1(0,L)$ let us define,
\begin{equation}
\label{eqn:bi} b_i(p_i,q_i)=\ds \int_0^L d_i \partial_l p_i
\partial_l q_i dl + \ds \int_0^L \gamma_i (\partial_l p_i)q_i dl
+\ds \int_0^L (z_i + \partial_l \gamma_i + \lambda )p_i q_i dl,
\end{equation}
\begin{equation}
\label{eqn:ci} c_i(\p,q_i)=-\ds \int_0^L (\M\p)_iq_i dl,
\end{equation}
\begin{equation}
\label{eqn:ei} e_i(\p,q_i)=b_i(p_i,q_i)+c_i(\p,q_i)
\end{equation}
Summing over $i$, we define for $\p,\q \in \H1$, the bilinear form
$e(\p,\q)$ by,
\begin{equation}
\label{eqn:e} e(\p,\q)=\ds \sum_{i=1}^N e_i(\p,q_i)
\end{equation}

\begin{lemma}
\label{lemmacoercive} For $\lambda > (\ds\frac{1}{2d^0}
||\gamma||^2_\infty +||\partial_l \gamma||^2_\infty +N
||\M||_\infty )$, the bilinear form $e(.,.)$ is continuous and
coercive on $\H1 \times \H1$, i.e there exist constants $C_1>0$
and $C_2>0$ such that
\begin{equation}
\label{continuous} |e(\p,\q)| \le C_1 ||\p||_{\H1} ||\q||_{\H1},\
\forall \p,\q \in \H1,
\end{equation}
\begin{equation}
\label{coercive} e(\p,\p) \ge C_2 ||\p||_{\H1}^2,\ \forall \p \in
\H1.
\end{equation}
\end{lemma}
\begin{proof}
Using Cauchy-Schwarz inequality we obtain,
$$
|\ds \sum_i b_i(p_i,q_i)| \le (||\d||_\infty +||\bgamma||_\infty
+||\z||_\infty +||\partial_l \bgamma||_\infty + \lambda)
||\p||_{\H1}||\q||_{\H1}
$$
and
$$
|\ds \sum_i c_i(\p,q_i)| = | \ds \sum_i \sum_j \int_0^L M_{ij} p_j
q_i dl| \le ||\M||_\infty N ||\p||_{\H1} ||\q||_{\H1}.
$$
which proves (\ref{continuous}).\\
Again using Cauchy-Schwarz inequality yields
$$
|\ds \int_0^L \gamma_i \partial_l p_i p_i|dl \le ||\gamma||_\infty
||\partial_l p_i||_{L^2(0,L)} ||p_i||_{L^2(0,L)}.
$$
Young's inequality then gives for any $\alpha >0$
$$
|\ds \int_0^L \gamma_i \partial_l p_i p_i dl | \le
\ds\frac{\alpha}{2}||\partial_l p_i||_{L^2(0,L)}^2 +
\ds\frac{1}{2\alpha} ||\gamma||^2_\infty ||p_i||^2_{L^2(0,L)}.
$$
Therefore we have that
$$
\ds \sum_i \int_0^L \gamma_i \partial_l p_i p_i dl \ge
-\ds\frac{\alpha}{2} ||\partial_l \p||_{\HH}^2 -
\ds\frac{1}{2\alpha} ||\gamma||^2_\infty ||\p||^2_{\HH}.
$$
Now since $\mu_i$ and $f_i$ are positive and $d_i$ is bounded
below by $d^0$, it follows that
$$
e(\p,\p) \ge (d^0 -\ds\frac{\alpha}{2})||\partial_l \p||_{\HH}^2
+(\lambda-(\ds\frac{1}{2\alpha} ||\gamma||^2_\infty +||\partial_l
\gamma||^2_\infty +N ||\M||_\infty ))||\p||^2_{\HH}.
$$
It is possible to choose $\alpha=d^0$ and $\lambda$ such that $
\lambda^0=(\lambda-(\ds\frac{1}{2d^0} ||\gamma||^2_\infty
+||\partial_l \gamma||^2_\infty +N ||\M||_\infty )) >0$ and
$C_2=min(\ds\frac{d^0}{2},\lambda^0)$.
\end{proof}

We can now give the definition of a weak solution to the
initial-boundary value problem (\ref{eqn:e1-1})-(\ref{eqn:e1-3})
and state the results which are shown in Section
\ref{section:exist} and \ref{section:positive}. \noindent A weak
solution to the initial-boundary value problem
(\ref{eqn:e1-1})-(\ref{eqn:e1-3})
is a vector valued function $\p$ satisfying the following {\bf problem (P)}:\\
Find
\begin{equation}
\label{P1} \p \in V,\ {\mathrm{such}}\ {\mathrm{that}}\ D\p \in
V',
\end{equation}
solution of
\begin{equation}
\label{P2} \ds \int_\O < D\p,\q>dtda+\ds \int_\O e(\p,\q)dtda=0,\
\forall \q \in V,
\end{equation}
\begin{equation}
\label{P3} \p(0,a,l)=\p^0(a,l) \quad {\mathrm{a.e}}\
{\mathrm{in}}\ (0,A)\times(0,L),
\end{equation}
\begin{equation}
\label{P4} \p(t,0,l)={\beta}(t,l,{\bf P}(t)) \quad {\mathrm{a.e}}\
{\mathrm{in}}\ (0,T)\times(0,L).
\end{equation}

In Section \ref{section:exist} it is proved that:
\begin{theorem}
\label{theo:exist-unique} There exists a unique solution $\p$ to
{\bf problem (P)}.
\end{theorem}

{\bf Notation}: $\p(t,a,l)$ and $\q(t,a,l)$ being vector valued
functions,
$\p \le \q$ means that $p_i \le q_i$ a.e. in $\Q$ for all $i \in [1:N]$.\\
With this notation, it is proved in Section \ref{section:positive}
that:
\begin{theorem}
\label{theo:comparaison}
The solution, $\p$, to {\bf problem (P)} is nonnegative a.e in $\Q$.\\
Moreover, let $\p^1$ (resp. $\p^2$) denote the solution to {\bf
problem (P)} associated with the vector of mortality rates $\z^1$
(resp. $\z^2$). If $\z^1 \le \z^2$ then $\p^2 \le \p^1$.
\end{theorem}

\section{Existence and uniqueness}
\label{section:exist} The proof of existence and uniqueness
consists in two main steps. First we show the result in the case
of a constant recruitment (independent of the fish density).
Second a fixed point argument enables to cope with the original
nonlinear recruitment.
\begin{lemma}
\label{lemma1} Let $\b$ be fixed in $L^2((0,T)\times(0,L))^N$.
There exists a unique $\p$ satisfying (\ref{P1})-(\ref{P3}) of
{\bf problem (P)} in which the initial condition (\ref{P4}) is
replaced by $\p(t,0,l)=\b(t,l)\ {\mathrm{a.e}}\ {\mathrm{in}}\
(0,T)\times(0,L)$.
\end{lemma}
\begin{proof}
The proof is an adaptation of the results given for the scalar
case in \cite{Langlais:1985}. We sketch it for the sake of
completeness. It consists in two steps.\\


{\bf Step 1}: We prove that given $\h \in V'$ there exists a
unique $\p \in V$, $D\p \in V'$ such that
\begin{equation}
\ds \int_\O <D\p,\q>dtda+\int_\O
e(\p,\q)dtda=\int_\O<\h,\q>dtda,\quad \forall \q \in V
\end{equation}
and $\p(0,a,l)=\p(t,0,l)=0$.\\

Let $A^0$ be the unbounded linear operator on
$(L^2(\Q))^N$ with domain $D(A^0)=\lbrace \p \in (L^2(\Q))^N,\
\partial_t \p + \partial_a \p \in (L^2(\Q))^N,
\p(0,t,l)=\p(t,0,l)=0 \rbrace$, defined by $\p \in D(A^0),\
A^0\p=\partial_t \p + \partial_a \p$. Then $-A^0$ is the
infinitesimal generator of a contraction semigroup, $(S(\tau)\p,\
\tau \ge 0)$, in $(L^2(\Q))^N$ (see \cite{Bardos:1970}) and
$$
(S(\tau)\p)(t,a,l)= \left \lbrace \begin{array}{l}
\p(t-\tau,a-\tau,l) \quad {\mathrm{if}}\ (t-\tau,a-\tau,l) \in \Q,\\
0 \quad {\mathrm{otherwise}}.\\
\end{array}
\right.
$$
From this one can deduce that the unbounded linear operator $A$
from $V$ to $V'$ with domain $D(A)=\lbrace \p \in V,\ D\p \in V',\
\p(0,a,l)=\p(t,0,l)=0 \rbrace$,
defined by $A\p = D\p$ is a maximal monotone operator.\\
With the bilinear form $e(.,.)$ we can define a linear bounded and
coercive operator $E$ from $V$ to $V'$ such that $<<E\p,\q>>=\ds
\int_\O e(\p,\q)dtda,\quad \forall \p,\q \in V$. Since $E$ is
bounded and coercive and $A$ is maximal monotone we conclude that
for any $\h \in V'$ there exists a unique $\p \in D(A)$ solution
to $A\p+E\p=\h$ which is an abstract formulation of our problem
because\\
$<<A\p,\q>>=\ds \int_\O <D\p,\q>dtda,\quad \forall \p \in D(A),\ \forall \q \in V$.\\

{\bf Step 2}: Let us now introduce a sequence of functions
$\bphi^n \in (C^\infty(\overline{\Q}))^N$ such that
$$
\begin{array}{l}
\bphi^n(0,a,l) \rightarrow \p^0(a,l)\quad {\mathrm{in}}\ (L^2((0,A)\times(0,L)))^N,\\
\bphi^n(t,0,l) \rightarrow \b(t,l)\quad {\mathrm{in}}\ (L^2((0,T)\times(0,L)))^N,\\
\end{array}
$$
From {\bf step 1}, we conclude that there exists a unique $\q^n$
in $D(A)$ solution to $A\q^n+E\q^n=-A\bphi^n-E\bphi^n$. Therefore
$\p^n=\q^n+\bphi^n$ is a solution to (\ref{P2}) satisfying
$\p^n(0,a,l)=\bphi^n(0,a,l)$ and
$\p^n(t,0,l)=\bphi^n(t,0,l)$.\\
Now taking $\p^n$ as a test function in (\ref{P2}), integrating by
parts using Lemma \ref{lemma:trace} and using the coercivity of
$e(.,.)$ we obtain that
$$
C_2 ||\p^n||^2_V \le \ds
\frac{1}{2}||\bphi^n(0,a,l)||^2_{(L^2((0,A)\times(0,L)))^N} +
\frac{1}{2}||\bphi^n(t,0,l)||^2_{(L^2((0,T)\times(0,L)))^N}.
$$
By the choice of $\bphi^n$ this implies that $\p^n$ is a bounded
sequence in $V$. Therefore we can extract a subsequence still
denoted $\p^n$ such that $\p^n \rightarrow \p$ weakly in $V$ and
$D\p^n \rightarrow \r$ weakly in $V'$. Since the operator $D$ is
continuous on ${\mathcal{D}}'(\O,(\H1)')$, $\r = D\p$. Moreover
since $E$ is continuous $E\p^n \rightarrow E\p$. We conclude that
$\p$ satisfies (\ref{P2}). The continuity of the trace
applications on $t=0$ and $a=0$ implies that $\p(0,a,l)=\p^0(a,l)$
and $\p(t,0,l)=\b(t,l)$.
\end{proof}

\begin{lemma}
\label{lemma2} Let
$C_3=(\underset{i\in[1:N]}{max}[AL^2||\bpsi||_\infty^2||\w||_\infty^2(\ds
\frac{e^{\lambda T}}{\theta_i})^2])^{1/2}$,
then the application\\
$(p_i(t,a,l)) \mapsto (\beta_i (t,l,P_i(t)))$ (cf Eqs \ref{eqn:e5}
and \ref{eqn:e6}) defines a bounded nonlinear operator, lipschitz
continuous from $L^2(\O,\HH)$ to $(L^2((0,T)\times(0,L))^N$ with
lipschitz constant $C_3$.
\end{lemma}
\begin{proof}
The application, $p_i(t,a,l)\mapsto P_i(t)=\ds \int_0^A
\int_{L_m}^Lw_i(t,a,l)p_i(t,a,l)da dl$, defines a bounded linear
operator from $L^2(\Q)$ to $L^2(0,T)$. This follows from,
$$
|\ds \int_0^A \int_{L_m}^Lw_i(t,a,l)p_i(t,a,l)da dl| \le \ds
\int_0^A \int_{0}^L|w_i(t,a,l)p_i(t,a,l)|da dl,
$$
and using Cauchy-Schwarz yields
$$
\ds \int_0^T|P_i(t)|^2dt \le ||\w||^2_\infty A L
||p_i||^2_{L^2(\Q)}.
$$
The application $p_i(t,a,l) \mapsto \beta_i (t,l,P_i(t))$ defines
a bounded nonlinear operator from $L^2(\Q)$ to
$L^2((0,T)\times(0,L))$. This follows from the fact that the
application $u_i(t,P)=\ds \frac{P}{\theta_i e^{-\lambda t} + |P|}$
from $[0,T]\times\mathbb{R}$ to $\mathbb{R}$ satisfies $|u_i(t,P)|
\le \ds \frac{e^{\lambda T}}{\theta_i}|P|$ and therefore we have
$$
\ds \int_0^T \int_0^L (\beta_i(t,l,P_i(t)))^2dtdl \le
||\bpsi||^2_\infty (\ds \frac{e^{\lambda T}}{\theta_i})^2
||\w||_\infty^2 A L^2 ||p_i||^2_{L^2(\Q)}.
$$
Lipschitz continuity follows from the fact that $(t,P)\mapsto
u_i(t,P)$ is lipschitz continuous in $P$ uniformly in $t \in
[0,T]$,
$$
|u_i(t,P^1) - u_i(t,P^2)| \le \ds \frac{e^{\lambda T}}{\theta_i}
|P^1 -P^2|,\quad \forall P^1,P^2 \in \mathbb{R},\ \forall t \in
[0,T].
$$
Hence, if to $p_i^1$ (resp. $p_i^2$) we associate $P_i^1$ (resp.
$P_i^2$) it holds that
$$
\begin{array}{l}
\ds \int_0^T \int_0^L [\beta_i(t,l,P_i^1(t))-\beta_i(t,l,P_i^2(t))]^2dtdl\\[7pt]
=\ds \int_0^T \int_0^L [\fc_{[0,L_b]}(l)\psi_i(t)(u_i(t,P_i^1(t))-u_i(t,P_i^2(t)))]^2dtdl,\\[7pt]
\le L||\bpsi||_\infty^2(\ds \frac{e^{\lambda T}}{\theta_i})^2 \ds \int_0^T |P_i^1(t)-P_i^2(t)|^2dt,\\[7pt]
\le AL^2||\bpsi||_\infty^2||\w||_\infty^2(\ds \frac{e^{\lambda
T}}{\theta_i})^2 ||p_i^1-p_i^2||_{L^2(\Q)}^2.
\end{array}
$$
\end{proof}

\begin{lemma}
\label{lemma3} There exists a unique $\p$ satisfying {\bf problem
(P)}
\end{lemma}
\begin{proof}
Let $\hat{\p}$ be given in $V$. With $\hat{\p}$ we associate a
vector $(\hat{P}_i(t))$. Let us denote $\mathcal{F}\hat{\p}=\p$
the solution to (\ref{P1})-(\ref{P3}) and satisfying
$(p_i(t,0,l))=(\beta_i(t,l,\hat{P}_i(t))$.\\
From Lemma \ref{lemma1} and Lemma \ref{lemma2} we deduce that the
nonlinear operator $\mathcal{F}$ maps $V$ into itself. Moreover it
follows from Lemma \ref{lemma:trace} that
$$
\ds \int_\O <D\p,\p> dtda \ge  - \ds \frac{1}{2} \int_0^A
||\p^0(a,.)||_{\HH}^2da - \ds \frac{1}{2} \int_0^T
||\p(t,0,.)||^2_{\HH}dt
$$
The coercivity of $e(.,.)$ leads to
$$
C_2 \ds \int_\O ||\p(t,a,.)||_{\H1}^2 dtda \le \ds \frac{1}{2}
\int_0^A ||\p^0(a,.)||_{\HH}^2da + \ds \frac{1}{2} \int_0^T
||\p(t,0,.)||^2_{\HH}dt.
$$
Lemma \ref{lemma2} then gives
$$
C_2 \ds \int_\O ||\p(t,a,.)||_{\H1}^2 dtda \le \ds \frac{1}{2}
\int_0^A ||\p^0(a,.)||_{\HH}^2da + \ds \frac{1}{2} C_3
||\hat{\p}||^2_{L^2(\O,\HH)},
$$
and $\mathcal{F}$ is bounded from $L^2(\O,\HH)$ to $V$.\\
The solutions we are looking for are the fixed points of
$\mathcal{F}$.
Let us show that $\mathcal{F}$ is a strict contraction in $L^2(\O,\HH)$.\\
Let $\hat{\p}^1$ and $\hat{\p}^2$ be given in $L^2(\O,\HH)$ and
let $\p^1=\mathcal{F}\hat{\p}^1$ and $\p^2=\mathcal{F}\hat{\p}^2$
be the associated solutions. The difference
$\p=\mathcal{F}\hat{\p}^1-\mathcal{F}\hat{\p}^2$ satisfies
(\ref{P1}),(\ref{P2}), $\p(0,a,l)=0$ and
$(p_i(t,0,l))=(\beta_i(t,l,\hat{P}_i^1(t))-\beta_i(t,l,\hat{P}_i^2(t)))$.\\
At the end of the proof of Lemma \ref{lemmacoercive}, since
$\lambda$ is arbitrary, one can choose $\lambda=\lambda_1 +
\lambda_2$ with $\lambda_1 >\ds\frac{1}{2d^0} ||\gamma||^2_\infty
+||\partial_l \gamma||^2_\infty +N ||\M||_\infty$ and $\lambda_2 >
0$ arbitrary. Hence,
$$
e(\p,\p)\ge \tilde{C_2}||\p||^2_{\H1} + \lambda_2 ||\p||^2_{\HH}
\ge \lambda_2 ||\p||^2_{\HH},\quad \forall \p \in \H1.
$$
Now using Lemma \ref{lemma:trace} once again we obtain
$$
\lambda_2 \ds \int_\O ||\p(t,a,.)||^2_{\HH}dtda \le \ds
\frac{1}{2} \sum_{i=1}^N \ds \int_0^T \int_0^L
[\beta_i(t,l,\hat{P}_i^1(t))-\beta_i(t,l,\hat{P}_i^2(t))]^2dtdl
$$
and Lemma \ref{lemma2} gives
$$
\lambda_2 ||\p||^2_{L^2(\O,\HH)} \le \ds \frac{1}{2} C_3
||\hat{\p}^1-\hat{\p}^2||_{L^2(\O,\HH)}.
$$
We can choose $\lambda_2 = C_3$ and since
$\p=\mathcal{F}\hat{\p}^1-\mathcal{F}\hat{\p}^2$ this proves that
$\mathcal{F}$ is a strict contraction on $L^2(\O,\HH)$. Thus it
follows from Banach fixed point theorem that $\mathcal{F}$ admits
a unique fixed point $\p$ which is the desired solution.
\end{proof}

\section{Positivity and comparison result}
\label{section:positive} In this section we first show in Lemma
\ref{lemma4} that the fish density population solution to our
model is positive. Then a comparaison result is given in Lemma
\ref{lemma5}.

\begin{lemma}
\label{lemma4} The solution $\p$ to {\bf problem (P)} is
nonnegative a.e. in $\Q$.
\end{lemma}
\begin{proof}
As in the proof of Lemma \ref{lemma3}, let $\hat{\p}$ be given in
$V$ and let $\mathcal{F}\hat{\p}=\p$ denote the solution to
(\ref{P1})-(\ref{P3}) and satisfying
$(p_i(t,0,l))=(\beta_i(t,l,\hat{P}_i(t)))$.
Let us also assume that $\hat{\p} \ge 0$.\\
The negative parts of $p_i(0,a,l)$ and $p_i(t,0,l)$ satisfy
$(p_i(0,a,l))^{-} = (p_i^0(a,l))^{-}=0$
and $(p_i(t,0,l))^{-} = ( \beta_i (t,l,\hat{P}_i(t)) )^{-} =0$.\\
One can then show using Lemma \ref{lemma:trace} (see
\cite{Langlais:1985}) that
$\ds \int_\O <D\p,\p^->dtda \le 0$.\\
The bilinear form $e$ can be decomposed as
$e(\p,\p^-)=e(\p^+,\p^-)-e(\p^-,\p^-)$,
with\\
$e(\p^+,\p^-)=\ds \sum_{i=1}^N b_i(p_i^+,p_i^-)+c_i(\p^+,p_i^-)$.\\
It holds that $b_i(p_i^+,p_i^-)=0$ since one can check that
$b_i(p_i,p_i^-)=-b_i(p_i^-,p_i^-)$. Moreover,
$c_i(\p^+,p_i^-)=-\ds\int_0^L \ds \sum_{j=1}^N M_{ij} p_j^+ p_i^-
dl \le 0$,
since $M_{ij} p_j^+ p_i^- \ge 0$ for $i\ne j$ and $M_{ii} p_i^+ p_i^-=0$.\\
We conclude that $e(\p^+,\p^-) \le 0$.\\
Taking $\q = \p^-$ in Eq. \ref{P2} yields,
$$
\ds \int_\O <D\p,\p^-> dtda+\int_\O e(\p^+,\p^-)dtda - \int_\O
e(\p^-,\p^-)dtda=0
$$
so that we obtain $\int_\O e(\p^-,\p^-)dtda \le 0$. The coercivity
of $e$ gives,\\
$C_2 ||\p^-||_V \le 0$, that is to say $\p$ is nonnegative.\\
If we define a sequence with $\p^1=\hat{\p}$ and
$\p^{n+1}=\mathcal{F}\p^n$, then from the previous lines we deduce
that $\p^n$ is nonnegative for all $n \ge 1$. By Banach fixed
point theorem this sequence converges to the solution $\p$ which
is therefore nonnegative.
\end{proof}

\begin{lemma}
\label{lemma5} Let $\p^1$ (resp. $\p^2$) denote the solution to
{\bf problem (P)} associated with the vector of mortality rates
$\z^1$ (resp. $\z^2$). If $\z^1 \le \z^2$ then $\p^1 \ge \p^2$.
\end{lemma}
\begin{proof}

{\bf Step 1}: Let $\hat{\p}^1$ and $\hat{\p}^2$ be given in $V$
and satisfying $0 \le \hat{\p}^1 \le \hat{\p}^2$. Let
$\p^1=\mathcal{F}\hat{\p}^1$ and $\p^2=\mathcal{F}\hat{\p}^2$ be
the associated solutions defined as in the proof
of Lemma \ref{lemma4}. Let us show that $\p^1 \le \p^2$.\\
It is clear that $\hat{P}_i^1(t) \le \hat{P}_i^2(t)$ a.e in
$(0,T)$, then since $u_i(t,P)$ is an increasing function of $P$ it
holds that $\beta_i(t,l,\hat{P}_i^1(t)) \le
\beta_i(t,l,\hat{P}_i^2(t))$ a.e in $(0,T)\times(0,L)$. The
difference $\p=\p^2-\p^1$ satisfies (\ref{P1}),(\ref{P2}) and
$$
\begin{array}{l}
\p(0,l,a)=0,\\
\p(t,0,a)=(\beta_i(t,l,\hat{P}_i^2(t)))-(\beta_i(t,l,\hat{P}_i^1(t)))
\ge 0.
\end{array}
$$
This is the same situation as in the first part of proof of Lemma
\ref{lemma4} and we conclude that $\p$ is nonnegative that is to
say $\p^1 \le \p^2$.

{\bf Step 2}: Let $\hat{\p} \ge 0$ be given in $V$. To the vectors
of mortality rates $\z^1$ and $\z^2$ ($0 \le \z^1 \le \z^2$) we
associate the bilinear forms $e^1$ and $e^2$ (see
(\ref{eqn:bi})-(\ref{eqn:ei}), note that $c_i^1(.,.)=c_i^2(.,.)$)
as well as the nonlinear operators $\mathcal{F}^1$ and
$\mathcal{F}^2$ defined as in the proof of lemma \ref{lemma4}.
They define the solutions
$\p^1=\mathcal{F}^1\hat{\p}$ and $\p^2=\mathcal{F}^2\hat{\p}$. Let us show that $\p^1 \ge \p^2$.\\
$(\p^2 - \p^1)$ satisfies
\begin{equation}
\label{eqn:raslebol1} \ds \int_O <D(\p^2-\p^1),\q>dtda+\int_\O
[e^2(\p^2,\q)-e^1(\p^1,\q)]dtda=0
\end{equation}
\begin{equation}
\label{eqn:raslebol2} (\p^2-\p^1)(0,a,l)=0
\end{equation}
\begin{equation}
\label{eqn:raslebol3} (\p^2-\p^1)(t,0,l)=0
\end{equation}
Let us choose $\q=(\p^2 - \p^1)^+$. From the equality $z_i^2 p_i^2
- z_i^1 p_i^1 = z_i^1(p_i^2 -p_i^1) + p_i^2(z_i^2-z_i^1)$ follows
that
\begin{equation}
\label{eqn:quitue}
\begin{array}{l}
e_i^2(\p^2,(p_i^2-p_i^1)^+)-e_i^1(\p^1,(p_i^2-p_i^1)^+)\\
=c_i((\p^2-\p^1)^+,(p_i^2-p_i^1)^+)-c_i((\p^2-\p^1)^-,(p_i^2-p_i^1)^+)\\
+b_i((p_i^2-p_i^1)^+,(p_i^2-p_i^1)^+)-b_i((p_i^2-p_i^1)^-,(p_i^2-p_i^1)^+)\\
+\int_0^L p_i^2(f_i^2 -f_i^1)(p_i^2-p_i^1)^+ dl
\end{array}
\end{equation}
We have already shown in the proof of Lemma \ref{lemma4} that
$c_i((\p^2-\p^1)^-,(p_i^2-p_i^1)^+) \le 0$ and that
$b_i((p_i^2-p_i^1)^-,(p_i^2-p_i^1)^+)=0$. Moreover since $\p^2$ is
nonnegative the last term of equality (\ref{eqn:quitue}) is
nonnegative. Then we obtain that
$$
e^2(\p^2,(\p^2-\p^1)^+)-e^1(\p^1,(\p^2-\p^1)^+) \ge
e^1((\p^2-\p^1)^+,(\p^2-\p^1)^+).
$$
Since $(\p^2-\p^1)$ satisfies (\ref{eqn:raslebol2}) and
(\ref{eqn:raslebol3}) it also holds that
$$
\ds \int_O <D(\p^2-\p^1),(\p^2-\p^1)^+>dtda \ge 0,
$$
so that
$$
\ds \int_\O e^1((\p^2-\p^1)^+,(\p^2-\p^1)^+) dtda \le 0,
$$
and using the coercivity of $e^1$ we finally obtain $\p^2 \le
\p^1$.

{\bf Step 3}: Let $\hat{\p} \ge 0$ be given in $V$. We define two
sequences $(\p^{1,n})_{n \ge 1}$ and $(\p^{2,n})_{n \ge 1}$ by
($\p^{1,1}=\hat{\p}$, $\p^{1,n+1}=\mathcal{F}^1\p^{1,n}$) and
($\p^{2,1}=\hat{\p}$, $\p^{2,n+1}=\mathcal{F}^2\p^{2,n}$).\\
From {\bf step 2} follows that $\p^{1,2} \ge \p^{2,2}$.\\
In addition to $\p^{1,3}=\mathcal{F}^1\p^{1,2}$ and
$\p^{2,3}=\mathcal{F}^2\p^{2,2}$,
let us define $\q^{3}=\mathcal{F}^2\p^{1,2}$.\\
The inequality $\p^{1,3} \ge \q^3$ follows from {\bf step 2},
whereas $\p^{2,3} \le \q^3$ follows from {\bf step 1}. Therefore
$\p^{1,3} \ge \p^{2,3}$. An induction then shows that $\p^{1,n}
\ge \p^{2,n},\ \forall n \ge 1$ and since the sequences converge
to the solution $\p^1$ and $\p^2$ of {\bf problem (P)} associated
with the vector of mortality rates $\z^1$ and $\z^2$ respectively,
the proof is complete.
\end{proof}

\section{Concluding remarks}

In this paper we have investigated a multi-region nonlinear
age-size structured fish population model. The model was
formulated in a generic way so that it can be potentially used for
various fish species. We formulated an initial boundary-value
problem and proved existence and uniqueness of a positive weak
solution. We also proved a comparison result which shows that the
variations
in the mortality rate in each region have consequences on the population of fish in every regions.\\
Other important problems need to be addressed now and are
currently under progress. The first one concerns the numerical
implementation of this model. In order to integrate numerically
system (\ref{eqn:e1-1})-(\ref{eqn:e1-3}) we use the characteristic
method. Indeed this system can be viewed as a collection of
systems of parabolic equations on the characteristic lines
$$S =\lbrace (t_0+s,a_0+s);\ s \in (0,s_{max}(t_0,a_0)) \rbrace,$$
where $(t_0,a_0) \in \lbrace 0 \rbrace \times (0,A) \cup (0,T) \times
\lbrace 0 \rbrace$. Each of these systems is then integrated in time with an operator
splitting method using the Lie formula (\cite{Strang:1968}, \cite{Marchuk:1990}).\\
The second problem concerns the estimation of the different badly
known parameters of the model (growth, mortality and migration
rates) from the data available for fisheries and mentioned in the
Introduction. In order to solve numerically this inverse problem,
the implementation of a variational data assimilation method is
under progress. The objective is to obtain a synthetic
representation of the real system combining theoritical knowledge
(the model) and experimental knowledge (the data).






\bibliographystyle{elsart-num}
\bibliography{biblioFauMau}

\begin{thebibliography}{10}
\expandafter\ifx\csname url\endcsname\relax
  \def\url#1{\texttt{#1}}\fi
\expandafter\ifx\csname urlprefix\endcsname\relax\def\urlprefix{URL }\fi

\bibitem{Megrey:1989}
B.~Megrey, Review and comparison of age-structured stock assessment models from
  theoritical and applied points of view, in: E.~Edwards, B.~Megrey (Eds.),
  Mathematical analysis of fish stocks dynamics, Vol.~6, AM. Fish. Soc. Symp.,
  1989, pp. 8--48.

\bibitem{Webb:1985}
G.~Webb, The theory of nonlinear age-dependent population dynamics, Marcel
  Dekker, 1985.

\bibitem{DeAngelis:1993}
D.~DeAngelis, K.~Rose, L.~Crowder, E.~Marschall, D.~Lika, Fish cohort dynamics:
  application of complementary modeling approaches, Am. Nat. 42 (1993)
  604--622.

\bibitem{Swart:1994}
J.~Swart, A.~Meijer, A simplified model for age-dependent population dynamics,
  Math. Biosci. 121 (1994) 15--36.

\bibitem{Arino:1995}
O.~Arino, A survey of structured cell population dynamics, Acta Biotheor. 43
  (1995) 3--25.

\bibitem{Maury:2004}
O.~Maury, B.~Faugeras, V.~Restrepo, {FASST}: A {F}ully {A}ge-{S}ize and
  {S}pace-{T}ime structured statistical model for the assessment of tuna
  populations, ICCAT Coll. Vol. Sci. Pap. in revision.

\bibitem{Langlais:1985}
M.~Langlais, A nonlinear problem in age-dependent population diffusion, SIAM.
  J. Math. Anal. 16~(3) (1985) 510--529.

\bibitem{Pfister:2002}
A.~Pfister, Some consequences of size variability in juvenile prickly sculpin,
  {\it cottus asper}, Environmental Biology of Fishes 66 (2002) 383--390.

\bibitem{Beverton:1996}
R.~Beverton, S.~Holt, On the {D}ynamics of of {E}xploited {F}ish {P}opulations,
  Fish and Fisheries Series 11, Chapman \& Hall, 1996.

\bibitem{Fournier:1990}
D.~A. Fournier, J.~R. Sibert, {MULTIFAN} a {L}ikelihood-{B}ased {M}ethod for
  {E}stimating {G}rowth {P}arameters and {A}ge {C}omposition from {M}ultiple
  {L}ength {F}requency {D}ata {S}ets {I}llustrated using {D}ata for {S}outhern
  {B}luefin {T}una ({\it {t}hunnus maccoyii}), Can. J. Fish. Aquat. Sci. 47
  (1990) 301--317.

\bibitem{Okubo:1980}
A.~Okubo, Diffusion and {E}cological {P}roblems: {M}athematical {M}odels,
  Vol.~10 of Biomathematics, Springer-{V}erlag, 1980.

\bibitem{Garroni:1982}
M.~Garroni, M.~Langlais, Age-{D}ependent {P}opulation {D}iffusion with
  {E}xternal {C}onstraint, J. Math. Biology 14 (1982) 77--94.

\bibitem{Lions:1968}
J.~Lions, E.~Magenes, Probl\`emes aux limites homog\`enes et applications,
  Paris, Dunod, 1968.

\bibitem{Bardos:1970}
C.~Bardos, Probl\`emes aux limites pour les \'equations aux d\'eriv\'ees
  partielles du premier ordre \`a coefficients r\'eels; th\'eor\`emes
  d'approximation; application \`a l'\'equation de transport., Ann. Scient. Ec.
  Norm. Sup.; $4^{eme}$ s\'erie 3 (1970) 185--233.

\bibitem{Strang:1968}
G.~Strang, On the construction and comparison of difference schemes, SIAM J.
  Numer. Anal. 5 (1968) 506--517.

\bibitem{Marchuk:1990}
G.~Marchuk, Splitting and alternating direction methods, in: Handbook of
  numerical analysis, Vol.~I, North-{H}olland, {A}msterdam, 1990, pp. 197--462.

\end{thebibliography}

\end{document}